\documentclass[11pt]{amsart}
\usepackage[a4paper, total={5.35 in, 8.3 in}]{geometry}
\usepackage{amsmath,amsfonts,amssymb}
\usepackage{amsrefs}
\usepackage{amsfonts}
\usepackage{verbatim}
\usepackage{enumerate}
\usepackage{amsthm}
\usepackage{url}
\usepackage[colorlinks]{hyperref}
\usepackage{tikz}

\newcommand{\lcm}{\mathrm{lcm}}
\newcommand{\Z}{\mathbb Z}

\newcommand{\C}{\mathcal C}

\newcommand{\F}{\mathbb{F}}
\newcommand{\fq}{\mathbb{F}_q}
\newcommand{\I}{\mathbb I}
\newcommand{\ord}{\mathrm{ord}}

\newtheorem{theorem}{Theorem}[section]
\newtheorem{conjecture}[theorem]{Conjecture}
\newtheorem{proposition}[theorem]{Proposition}
\newtheorem{lemma}[theorem]{Lemma}
\newtheorem{corollary}[theorem]{Corollary}

\theoremstyle{definition}
\newtheorem{definition}[theorem]{Definition}
\newtheorem{problem}[theorem]{Problem}
\newtheorem{example}[theorem]{Example}

\theoremstyle{remark}
\newtheorem{remark}[theorem]{Remark}

\author[S.D. Cohen]{Stephen D. Cohen}\thanks{The first author is Emeritus Professor of Number Theory, University of Glasgow}
\address{6 Bracken Road, Portlethen, Aberdeen AB12 4TA, Scotland, UK}
\email{Stephen.Cohen@glasgow.ac.uk}

\author[G. Kapetanakis]{Giorgos Kapetanakis}
\address{Department of Mathematics, University of Thessaly, 3rd km Old National Road Lamia-Athens, 35100 Lamia, Greece}
\email{gnkapet@gmail.com}

\author[L. Reis]{Lucas Reis}
\address{Departamento de Matem\'{a}tica,
Universidade Federal de Minas Gerais,
UFMG,
Belo Horizonte MG (Brazil),
 31270901}
\email{lucasreismat@gmail.com}

\title{The existence of $\F_q$-primitive points on curves using  freeness}

\date{\today
}


\begin{document}

\begin{abstract}
Let  $\C_Q$  be the   cyclic group of order $Q$, $n$ a divisor of $Q$ and $r$ a divisor of $Q/n$.  We introduce the set of $(r,n)$-free elements of $\C_Q$  and derive a   lower bound for the the number of elements  $\theta \in \F_q$ for which $f(\theta)$ is $(r,n)$-free  and $F(\theta)$ is $(R,N)$-free, where $ f, F \in \F_q[x]$. As an application, we consider the existence of $\F_q$-primitive points   on curves like $y^n=f(x)$ and find, in particular, all the odd prime powers $q$ for which the  elliptic curves  $y^2=x^3 \pm x$  contain an $\F_q$-primitive point.
\end{abstract}
\maketitle

\section{Introduction}
For a prime power $q$, let $\F_q$ be the finite field with $q$ elements. It is well-known that the multiplicative group $\F_q^*=\F_q\setminus \{0\}$ is cyclic and any generator of such a group is a \emph{primitive} element of $\F_q$. Primitive elements are a recurrent object of study in the finite field theory, mainly because of their applications in practical situations such as the discrete logarithm problem. Vinogradov obtained a simple character sum formula for the indicator (characteristic) function of such elements~\cite{mullenpanario13}*{Theorem~6.3.90}. The latter can be subsumed into a  general concept of freeness, which is strongly related to the multiplicative structure of the elements of $\F_q^*$. More precisely, for a divisor $d$ of $q-1$, an element of $\F_q$ is \emph{$d$-free} if is not of the form $\beta^s$ for any divisor $s>1$ of $d$. Evidently,  primitive elements of $\F_q$ are just the $(q-1)$-free elements of $\F_q$.

From a theoretical point of view, many authors have explored the existence and number of primitive elements of finite fields with additional properties. The main tools are Vinogradov's formula and bounds on multiplicative character sums such as Weil's bound. A common theme is the description of  finite fields containing a pair $(\alpha, F(\alpha))$ of primitive elements of $\F_q$, where $F\in \F_q(x)$ is a rational function. The latter is equivalent to looking at $\F_q$-rational points on the $0$-genus curve $\C:\, y=F(x)$ whose coordinates are primitive.   Such a point will be referred to as an  {\em $\F_q$-primitive point}. Relevant articles containing relatively complete results on the  existence of $\F_q$-primitive points on a curve $y=F(x)$  include \cite{cohenoliveiraesilvastrudgian15} (for $F$ a general linear polynomial), \cite{cohenoliveiraesilvasutherlandtrudgian18}, Corollary 2 (i),(ii) (for $F(x) = x \pm 1/x$) and  \cite{bookercohensutherlandtrudgian19} (for $F$ a general quadratic polynomial). Additionally,   \cite {carvalhoguardiieironeumanntizziotti21}  applies to  rational functions  $F=f_1/f_2$ where $f_1$ and $f_2$ are polynomials, with partial numerical  results for $\deg f_1+ \deg f_2 \le 7$. 
A natural extension  would  be to consider the existence of $\F_q$-primitive points on curves of the form  $y^n=F(x) $, where  $n$ is an integer  indivisible by the field characteristic and $F$ is a rational function in  $\F_q(x)$. An important example would be  that of  elliptic curves, $y^2=f(x)$,  where $q$ is an odd prime power and  $f$ is a square-free cubic 
polynomial. (Note that in this last situation, our terminology is to be distinguished from that of a primitive point introduced in  \cite{langtrotter77}.)

In this paper we generalize the notion of freeness, also considering the more general setting of finite cyclic groups. Such a concept not only recovers the former description for primitive elements but also the description of elements in $\F_q^*$ with any prescribed multiplicative order. In particular, we obtain a character sum formula for the indicator function of elements in $\F_q^*$ with prescribed multiplicative order, recovering a result from Carlitz~\cite{carlitz52}.

Next,  we extend the idea of freeness to  the  definition of $(r,n)$-free elements in a finite cyclic group (introduced in Section \ref{Sec:(r,n)-free}).   This  is an  appropriate one for the  discussion of the existence of primitive points on curves of the form $y^n=F(x)$ and more general questions.  In this context,  we then study pairs of polynomial expressions with special restrictions, obtaining a criterion for the existence of such pairs (Corollary \ref{cor:main_cond}). As an application of the latter (and  a ``sieving'' version, Theorem \ref{thm:main_siev}), we obtain both asymptotic and concrete results on the existence of $\F_q$-primitive points of the elliptic curve $\mathcal C:\, y^2=f(x)$.  These are especially effective when studying elliptic curves of the form $\C_a: y^2=x^3-ax, a \in \F_q^*$.  In particular, we shall establish the following theorem.
 \begin{theorem}\label{thm:ell1}
  Let $q$ be an odd prime power.   Then there exist an $\F_q$-primitive point on the elliptic curve $\C_1$ if and only if $q \notin \{ 3, 7, 13, 17, 25, 49,  121\}$.

 Similarly, there exists an $\F_q$-primitive point on the elliptic curve $\C_{-1}$ if and only if $q \notin \{5,9, 17, 41, 49\}$.
 
 \end{theorem}
 
 More generally,  after our theoretical work and calculations, we are enabled to make the following conjecture. 

\begin{conjecture}\label{conj:ell2}
Let $q$ be an odd prime power.

 Suppose $q \notin S:= \{3,5,7,9,13,17,25, 29,31,41,49, 61,73,81, 121,337\}$.  Then, for any $a \in \F_q^*$, there exists an $\F_q$-primitive point on $\C_a$.  
\end{conjecture}
We know that Conjecture~\ref{conj:ell2} holds for prime powers $q$ outside the interval [141121, 167763671] (see Section 4). With reference to Conjecture \ref{conj:ell2}, for any $q \in S$,  the precise  number of curves $\C_a$ (values of $a \in \F_q^*$) which do {\em not} contain an $\F_q$-primitive point is tabulated in Table \ref{tab:1} below (at the end of Section 4).

\section{Preparation}
This section provides some background material that is further used. We start fixing some notation. For positive integers $a$ and $b$, we set  $a_{(b)}=\frac{a}{\gcd(a, b)}$. As usual, $\mu$ and $\phi$ stand for the M\"obius and Euler totient function, respectively. Moreover, for a positive integer $A$, we denote its square-free part by $A^*$.
 
 \subsection{Characters}
Recall that, for a finite group $G$, a \emph{character} of $G$ is a group homomorphism $\eta:G\to \mathbb C^{*}$. If $G$ is cyclic of order $n$ with generator $g$, any character of $G$ is uniquely determined by the image of $g$. Moreover, since $g^n=1$, such an image must be an $n$-th complex root of unity. From these observations, we readily obtain the following result.

\begin{lemma}
If $G$ is a cyclic group of order $n$ with generator $g$, the set of characters of $G$ is a multiplicative group of order $n$, generated by the character $\eta: g^k\mapsto e^{\frac{2\pi  ik}{n}}$.
\end{lemma}

The map $g\mapsto 1\in \mathbb C$ is always a character of $G$, commonly called the \emph{trivial} character of $G$. Given a character $\eta$ of $G$, the least positive integer $k$ such that $\eta(h)^k=1$ for every $h\in G$ is the \emph{order} of $\eta$, denoted by $\ord(\eta)$. 

\begin{definition}
If $\F_q$ is a finite field, a {\em multiplicative character} of $\F_q$ is a character $\eta$ of the multiplicative cyclic group $G=\F_q^*$. We extend $\eta$ to $0\in \F_q$ by setting $\eta(0)=0$.
\end{definition}

The following well-known theorem provides a bound on character sums over finite fields with polynomial arguments.

\begin{theorem}[\cite{LN}*{Theorem~5.41}] \label{Gauss1}
Let $\eta$ be a multiplicative character of $\F_{q}$ of order $r>1$ and $F\in \F_{q}[x]$ be a polynomial of positive degree such that $F$ is not of the form $ag(x)^r$ for some $g\in \F_{q}[x]$ with degree at least $1$ and $a\in \F_q$. Suppose that $z$ is the number of distinct roots of $F$ in its splitting field over $\F_{q}$. Then the following holds: 
$$\left|\sum_{c\in \F_{q}}\eta(F(c))\right|\le (z-1)\sqrt{q}.$$
\end{theorem}

 \subsection{On $n$-primitive elements}
If $n\mid q-1$, then an element of $\fq$ of order $(q-1)/n$ is called \emph{$n$-primitive} and recently these elements have started attracting attention \cite{cohenkapetanakis19b,cohenkapetanakis19,cohenkapetanakis20,kapetanakisreis18}. This is mainly due to their theoretical interest and partially because, unlike primitive elements, we have efficient algorithms that locate such elements \cite{gao99,martinezreis16,popovych13}. A challenging aspect of their study is their characterization.
According to Carlitz~\cite{carlitz52}, we have the following result.

\begin{lemma}\label{lem:car}
If $N$ is a divisor of $q-1$, the characteristic function for the set of elements in $\F_{q}$ with multiplicative order $N$ can be expressed as
\begin{equation}\label{eq:Carlitz}\mathcal O_{N}(\omega)=\frac{N}{q-1}\sum_{d|N}\frac{\mu(d)}{d}\sum_{\ord(\eta)|\frac{d(q-1)}{N}}\eta(\omega).\end{equation}
\end{lemma}

By reordering of the terms in Eq.~\eqref{eq:Carlitz}, a standard  number-theoretic argument based on Möbius inversion yields the following alternative formula:
\begin{equation}\label{eq:Carlitz2}\mathcal O_N(\omega)=\frac{\phi(N)}{N}\sum_{t|q-1}\frac{\mu(t_{(n)})}{\phi(t_{(n)})}\sum_{\ord(\eta)=t}\eta(w),\quad n=\frac{q-1}{N}.\end{equation}


Note that the above expression of the characteristic function for $n$-primitive elements is in fact a generalization of Vinogradov's expression for the characteristic function for primitive elements \cite{mullenpanario13}*{Theorem~6.3.90}. Further, note that a similar variation of Vinogradov's formula, that characterizes $n$-primitive elements is proven in \cite[Lemma~2.1]{cohen03}. We omit the proof of the latter since a more general result is proved in Section~\ref{Sec:(r,n)-free}; see Remark~\ref{rem:main} and Proposition~\ref{prop:charI} for more details.



We end this section with an identity related to the sum appearing in Lemma~\ref{lem:car} that is further used.

\begin{lemma}\label{lem:arithfun}
For positive integers $r, n$, we have that
$$T(r, n):=\sum_{t\mid r} \frac{|\mu(t_{(n)})|}{\phi(t_{(n)})}\cdot\phi(t)=\gcd(r, n)\cdot W\left(\gcd(r,r_{(n)})\right),$$
where $W(a)$ denotes the number of square-free divisors of $a$. 
\end{lemma}
\begin{proof}
Observe that $f_n(r):=\frac{|\mu(r_{(n)})|}{\phi(r_{(n)})}\cdot\phi(r)$ and $g_n(r):=\gcd(r, n)\cdot W\left(\gcd(r,r_{(n)})\right)$ are multiplicative functions on $r$. In particular, the same holds for $T(r, n)$ and so it suffices to prove the equality $T(r, n)=g_n(r)$ in the case where $r$ is a prime power. Write $r=s^a$ with $s$ prime and $a\ge 1$, and write $n=s^b\cdot n_0$, where $b\ge 0$ and $\gcd(n_0, s)=1$. We split the proof into two  cases according as $a >b$ or $a\leq b$.

\begin{enumerate}[(i)]
\item If $a >b$, then
\begin{eqnarray*}
T(r, n)&=&\sum_{i=0}^b\frac{|\mu(1)|}{\phi(1)}\phi(s^i)+\sum_{i=b+1}^a\frac{|\mu(s^{i-b})|}{\phi(s^{i-b})}\phi(s^i) \\
&=& s^b+ \frac{\phi(s^{b+1})}{\phi(s)}=2s^b .
\end{eqnarray*}
Moreover, since  $\gcd (r,n)= s^b$ and $\gcd (r,r_{(n)})= s^{a-b}$ then $g_n(r)=2s^b=T(r,n)$.
\item If $a \leq b$, we have that 
$$T(r, n)=\sum_{i=0}^{a}\frac{|\mu(1)|}{\phi(1)}\cdot \phi(s^i)=s^a=r.$$
Moreover, in this case, $\gcd(r, n)=r$ and $r_{(n)}=1$, implying $g_n(r)=r$. 
\qedhere
\end{enumerate}
\end{proof}

\section{Introducing \( (r,n)\)-free elements}\label{Sec:(r,n)-free}
Motivated by the characterization of $n$-primitive elements in Eq.~\eqref{eq:Carlitz2}, we will generalize the well-known notion of $r$-free elements, considering also the more general setting of cyclic groups.

\begin{definition}\label{deffree}
Let $\C_{Q}$ be a multiplicative cyclic group of order $Q$. For a divisor $n$ of $Q$ and a divisor $r$ of $Q/n$, an element $h\in \C_Q$ is \emph{$(r, n)$-free} if the following hold:
\begin{enumerate}[(i)]
 \item $\ord(h)|\frac{Q}{n}$, i.e., $h$ is in the subgroup $\C_{Q/n}$; 
 \item $h$ is $r$-free in $\C_{Q/n}$, i.e., if $h=g^s$ with $g\in \C_{Q/n}$ and $s|r$, then $s=1$. 
\end{enumerate}
\end{definition}

In the following remark we present some straightforward facts on $(r, n)$-free elements.
  \begin{remark}\label{rem:main}
For a divisor $n$ of $Q$ and a divisor $r$ of $Q/n$, the following hold:
\begin{enumerate}[(i)] 
 \item  $(r, 1)$-free elements in $\C_Q$ are just the usual $r$-free elements;
  \item the $(Q/n,n)$-free elements in $\C_Q$ are exactly the elements of order $Q/n$.
 \end{enumerate}
  \end{remark}

   The following is a generalization of \cite[Proposition~5.2]{huczynskamullenpanariothomson13} and its proof is a mere adaptation of the original proof in our setting. We add it here, with the intention of making the relation between $(r,n)$-freeness and the multiplicative order of an element clear.
  \begin{lemma} \label{lem:rnfree}
  Let $n$ be a divisor of $Q$ and $r$ a divisor of $Q/n$. Then an element $h\in \C_Q$ is $(r, n)$-free if and only if $h=g^{n}$ for some $g\in \C_Q$ but $h$ is not of the form $g_0^{np}$ with $g_0\in \C_Q$, for every prime divisor $p$ of $r$. In particular, $h\in \C_Q$ is $(r, n)$-free if and only if  
$\gcd\left(rn, \frac{Q}{\ord(h)}\right)=n$.  
  \end{lemma}
  \begin{proof}
  Since the set $\{g^n\,|\, g\in \C_Q\}$ describes the elements in $\C_Q$ whose order divides $Q/n$, the first statement follows directly by the definition of $(r, n)$-free elements. For the second statement, pick $h$ an arbitrary element of $\C_Q$. One can easily verify that there exists a generator $g$ of $\C_Q$ and a divisor $t$ of $Q$ such that $h=g^t$. In this case, we have that $h$ is $(r, n)$-free if and only if $t=n\cdot s$, where $s$ divides $Q/n$ and $\gcd(s, r)=1$. We observe that the order of $h$ equals $\frac{Q}{t}=\frac{Q}{ns}$ and so 
  $$\gcd\left(rn, \frac{Q}{\ord(h)}\right)=n\cdot \gcd(r, s),$$
 from where the result follows.
  \end{proof}
 
  The following is an obvious consequence of Lemma~\ref{lem:rnfree}.
  \begin{lemma} \label{lem:rn}
    Let $n$ be a divisor of $Q$ and $r$ a divisor of $Q/n$. 
    An element of $\C_Q$ is $(r,n)$-free if and only if it is $(r^*,n)$-free.
    \end{lemma}
    
    It follows from Lemma \ref{lem:rn} that, wherever it is convenient, we may assume that $r$ is square-free.

Our next aim (see Proposition~\ref{prop:charI}) is to prove that
  \[ \I_{r,n}(h) := \frac{\phi(r)}{rn} \sum_{t\mid rn} \frac{\mu(t_{(n)})}{\phi(t_{(n)})} \sum_{\ord(\eta)=t} \eta(h) , \ h\in\C_Q.  \]
is a character-sum expression of the characteristic function for $(r,n)$-free elements of $\C_Q$.
   Note that this expression of the characteristic function is in fact a generalization of Vinogradov's expression of the characteristic function for $r$-free elements. In order to proceed with the proof, we will need the following lemma.
  \begin{lemma} \label{lemma:orth}
  Let $t\mid Q$ and $h\in \C_Q$. Then
  \[ S_0:=\frac{1}{t}\sum_{\ord(\eta) \mid t} \eta(h)=1, \]
  if $h=g^t$ for some $g\in \C_Q$. Otherwise, $S_0=0$.
  \end{lemma}
  \begin{proof}
  Immediate from the orthogonality relations; see Section~5.1 of~\cite{LN}.
  \end{proof}

  \begin{proposition} \label{prop:charI}
  Let $n$ be a divisor of $Q$ and $r$ a divisor of $Q/n$.  If $h\in\C_Q$, then \[ \I_{r,n}(h) = \begin{cases} 1,& \text{if $h$ is $(r,n)$-free}, \\ 0, &\text{otherwise}. \end{cases} \]
  \end{proposition}
  \begin{proof}
   Let $p_1,\ldots ,p_\kappa$ be the distinct prime divisors of $r$. From Lemma~\ref{lem:rnfree}, we obtain that $h\in\C_Q$ is $(r,n)$-free if and only if $h=g^n$ for some $g\in\C_Q$, but $h$ is not of the form $g_0^{np_i}$ with $g_0\in\C_Q$, for every $1\leq i\leq \kappa$. Further, if $I_k$ stands for the characteristic function of the set $\{ g^k \ : \ g\in\C_Q\}$, we have that $I_{np_i}I_n = I_{np_i}$. If $I_{r,n}$ is the characteristic function for $(r,n)$-free elements of $\C_Q$, it follows that, for every $h\in\C_Q$
 \begin{align}
 I_{r,n}(h) & = I_n(h) \prod_{i=1}^\kappa (1-I_{np_i}(h)) = \sum_{m\mid r} \mu(m) I_{nm}(h) \nonumber \\ & = \frac 1n \sum_{m\mid r} \frac{\mu(m)}{m} \sum_{\ord(\eta)\mid mn} \eta(h) = \frac 1n \sum_{m\mid r} \sum_{t\mid mn}\frac{\mu(m)}{m} \mathfrak{X}_t(h), \label{eq:I1}
 \end{align}
 where $\mathfrak{X}_t := \sum_{\ord(\eta)=t} \eta(h)$ and in the second-to-last equality we use Lemma~\ref{lemma:orth}.
 In order to obtain a more convenient expression of the above, let $n=p_1^{n_1}\cdots p_k^{n_k}$ and $r=p_1^{r_1}\cdots p_k^{r_k}$ be the prime factorization of $n$ and $r$, where $p_1,\ldots ,p_k$ are distinct primes and $n_i,r_i\geq 0$. Further, observe, that for any arithmetic functions $f$ and $g$, we have that
 \begin{multline*}
 \sum_{m\mid r} \sum_{t\mid mn} f(m)g(t) = \sum_{\substack{1\leq i\leq k \\ 0\leq m_i\leq r_i}} \ \sum_{\substack{1\leq i\leq k \\ 0\leq t_i\leq n_i+m_i}} f(p_1^{m_1} \cdots p_k^{m_k}) g (p_1^{t_1} \cdots p_k^{t_k}) \\
  \sum_{\substack{1\leq i\leq k \\ 0\leq t_i\leq n_i+ r_i}} \ \sum_{\substack{1\leq i\leq k \\ \max(0,t_i-n_i) \leq m_i \leq r_i}} f(p_1^{m_1} \cdots p_k^{m_k}) g (p_1^{t_1} \cdots p_k^{t_k}) = \sum_{t\mid rn} \sum_{m\mid \frac{r}{t_{(n)}}} f(t_{(n)}m) g(t) .
 \end{multline*}
 We use the latter in Eq.~\eqref{eq:I1} and obtain
 \[ I_{r,n}(h)  = \frac 1n \sum_{t\mid rn} \mathfrak{X}_t \sum_{m\mid (r/t_{(n)})} \frac{\mu(t_{(n)}m)}{t_{(n)}m} . \]
 %
 Note that the presence of the M\"obius function of the inner sum of the above, implies that only the square-free divisors $m$ of $r/t_{(n)}$ that are relatively prime with $t_{(n)}$ contribute to sum, i.e., we may rewrite the above as
 \[ I_{r,n}(h) = \frac 1n \sum_{t\mid rn} \mathfrak{X}_t \sum_{m\mid r_{t,n}} \frac{\mu(t_{(n)}m)}{t_{(n)}m} , \]
 where $r_{t,n}$ is the product of the the prime factors of $r/t_{(n)}$ that do not divide $t_{(n)}$. In particular, notice that $r_{t,n}t_{(n)}$ has exactly the same prime factors with $r$. Now, since $\mu(x)/x$ is multiplicative, we may rewrite the latter displayed equation as follows:
 \[ I_{r,n}(h) = \frac 1n \sum_{t\mid rn} \mathfrak{X}_t \frac{\mu(t_{(n)})}{t_{(n)}} \sum_{m\mid r_{t,n}} \frac{\mu(m)}{m} = \frac 1n \sum_{t\mid rn} \mathfrak{X}_t \frac{\mu(t_{(n)})}{t_{(n)}} \frac{\phi(r_{t,n})}{r_{t,n}} , \]
 where we used the well-known identity $\sum_{d\mid a}\frac{\mu(d)}{d} = \frac{\phi(a)}{a}$. Then we rewrite the above as follows
 \[ I_{r,n}(h) = \frac 1n \sum_{t\mid rn} \mathfrak{X}_t \frac{\mu(t_{(n)})}{\phi(t_{(n)})} \frac{\phi(r_{t,n}t_{(n)})}{r_{t,n}t_{(n)}} = \frac 1n \sum_{t\mid rn} \mathfrak{X}_t \frac{\mu(t_{(n)})}{\phi(t_{(n)})} \frac{\phi(r)}{r} . \]
 The result follows after replacing $\mathfrak{X}_t$ and accordingly rearranging the order in the above expression.
  \end{proof}
  
 \begin{remark} 
 If $\C_{Q}=\F_{q}^*$, recall that we extended the multiplicative characters $\eta$ of $\F_q^*$ to $0\in \F_q$ with $\eta(0)=0$. Within this extension the equality $\I_{r, n}(0)=0$ holds, as expected.    
\end{remark}

\section{On $(r, n)$-freeness through polynomial values}
For polynomials $f, F\in \F_q[x]$, we study the number of pairs $(f(y), F(y))$ such that $f(y)$ is $(r, n)$-free and $F(y)$ is $(R, N)$-free with $y\in \F_q$. It is only interesting to explore the case where $q-1$ has proper divisors, so we may assume that $q\ge 5$. Of course, this number of pairs can be zero if $f$ and $F$ have certain multiplicative dependence with respect to the numbers $rn$ and $RN$. The following example gives an instance of the latter.


\begin{example}\label{ex:if}
Let $q$ be an odd prime power,  $n=N=1$ and $r=R=2$. Suppose that $f, F\in \F_q[x]$ are non constant polynomials such that $f\cdot F$ is of the form $\lambda g^2$ with $\lambda$ a nonsquare of $\F_q$. In particular, there is no $y\in \F_q$ such that $f(y)$ and $F(y)$ are $(2, 1)$-free since this would imply that $f(y)\cdot F(y)$ is a nonzero square in $\F_q$.
\end{example}

We can avoid pathological situations like the one in Example~\ref{ex:if} by imposing the following mild condition: $f, F\in \F_q[x]$ are non constant square-free polynomials such that $f/F$ is not a constant. The next theorem shows that this condition asymptotically guarantees the existence of polynomials values $(f(y), F(y))$ with prescribed freeness.

\begin{theorem}\label{thm:main_cond}
Fix $q\ge 5$ a prime power, let $n, N$ be divisors of $q-1$ and let $r$ and $R$ be divisors of $\frac{q-1}{n}$ and $\frac{q-1}{N}$, respectively. Let $f, F\in \F_q[x]$ be non constant square-free polynomials such that the ratio $f/F$ is not a constant and let $D+1\ge 2$ be the number of distinct roots of $f\cdot F$ over its splitting field. Then the number $N_{f, F}=N_{f, F}(r, n, R, N)$ of  elements $\theta\in \F_{q}$ such that $f(\theta)$ is $(r,n)$-free and $F(\theta)$ is $(R, N)$-free  satisfies
$$N_{f, F}=\frac{\phi(r)\phi(R)}{rnRN}\left(q+H(r, n, R, N)\right),$$
with $|H(r, n, R, N)|\le DnNW(r)W(R)q^{1/2}$.
\end{theorem}

\begin{proof}
By definition, we have that $N_{f, F}=\sum_{w\in \F_q}\I_{r, n}(f(w))\cdot \I_{R, N}(F(w))$. From Proposition~\ref{prop:charI}, for $\delta=\frac{\phi(r)\phi(R)}{rnRN}$ we have that 
\begin{align*}\frac{N_{f,F}}{\delta} & =\sum_{w\in \F_q}\left(\sum_{t\mid nr} \frac{\mu(t_{(n)})}{\phi(t_{(n)})} \sum_{\ord(\eta)=t} \eta(f(w))\right)\cdot \left(\sum_{T\mid RN} \frac{\mu(T_{(N)})}{\phi(T_{(N)})} \sum_{\ord(\chi)=T}\chi(F(w))\right)
\\ {} & = \sum_{t|rn, \; T|RN}\frac{\mu(t_{(n)})\cdot \mu(T_{(N)})}{\phi(t_{(n)})\cdot \phi(T_{(N)})}\sum_{\ord(\eta)=t\atop \ord(\chi)=T}G_{f, F}(\eta, \chi),\end{align*}
where $G_{f, F}(\eta, \chi)=\sum_{w\in \F_{q}}\eta(f(w))\cdot \chi(F(w))$. Fix $t|rn$ and $T|RN$ and let $\eta, \chi$ be multiplicative characters of $\F_q$ with orders $t$ and $T$, respectively. Then $$\eta(w)\cdot \chi(f(w))=\tilde{\eta}\left(f(w)^{cL/t}\cdot F(w)^{CL/T}\right),$$ for some multiplicative character $\tilde{\eta}$ of $\F_q$ order $L=\lcm(t, T)$ and some integers $1\le c\le t$ and $1\le C\le T$ with $\gcd(c, t)=\gcd(C, T)=1$. Since $f, F$ are square-free and the ratio $f/F$ is not a constant, we easily verify that the polynomial $\mathcal F(x)=f(x)^{cL/t}\cdot F(x)^{CL/T}$ is of the form $\kappa\cdot g(x)^{L}$ if and only if $t=T=1$. Therefore, from Theorem~\ref{Gauss1}, we have that $|G_{f, F}(\eta, \chi)|\le Dq^{1/2}$ whenever $(t, T)\ne (1, 1)$. For $t=T=1$, we observe that $\eta$ and $\chi$ are just the trivial multiplicative character over $\F_q$ and so $G_{f, F}(\eta, \chi)=q-\epsilon$, where $\epsilon$ is the number of roots of $f\cdot F$ over $\F_q$. Since $\epsilon\le D+1$, applying estimates we obtain that
$$\left|\frac{N_{f,F}}{\delta}- q\right|\le D+1+Dq^{1/2}\cdot M,$$
where 
$$M=\sum_{t|rn,\;  T|RN\atop (t, T)\ne (1, 1)}\frac{|\mu(t_{(n)})\cdot \mu(T_{(N)})|}{\phi(t_{(n)})\cdot \phi(T_{(N)})}\sum_{\ord(\eta)=t\atop \ord(\chi)=T}1=T(rn, n)\cdot T(RN, N)-1,$$
and $T(a, b)$ is as in Lemma~\ref{lem:arithfun}. According to Lemma~\ref{lem:arithfun}, we have the equality $T(ab, b)=b\cdot W(a)$ and so 
$$\left|\frac{N_{f,F}}{\delta}-q\right|\le D+1+Dq^{1/2}(nNW(r)W(R)-1)< DnNW(r)W(R)q^{1/2},$$
where in the last inequality we used the fact that $D+1-Dq^{1/2}< 0$ if $D\ge 1$ and $q\ge 5$. The proof is complete.
\end{proof}
We immediately obtain the following corollary:
\begin{corollary} \label{cor:main_cond}
Let $q, r,R,n,N,f,F$ and $D$ be as in Theorem~\ref{thm:main_cond}. If
\[ q^{1/2} > DnNW(r)W(R) , \]
then $N_{f,F}(r, n, R, N)>0$.
\end{corollary}

\begin{remark}
Following the proof of Theorem~\ref{thm:main_cond}, one may check that the condition we impose on $f, F$ can be replaced by a less restrictive one. In fact it suffices to assume that, for every $t\mid rn$ and $T\mid RN$ with $(t, T)\ne (1, 1)$, and integers $1\le c\le t$ and $1\le C\le T$ with $\gcd(t, c)=\gcd(T, C)=1$, the polynomial $f^{cT}F^{Ct}$ is not the form $\kappa \cdot g^{tT}$ with $g\in \F_q[x]$ and $\kappa\in \F_q$. We observe that if both $f$ and $F$ have one simple linear factor (distinct from each other), then the latter always holds and such polynomials are not necessarily square-free.
\end{remark}





\subsection{The prime sieve}

The aim of the section is to relax  further the condition of Theorem~\ref{thm:main_cond}. For this reason, we will employ the Cohen-Huczynska sieving technique, \cite{cohenhuczynska03}.

We describe the prime sieve in the context of the  the general cyclic group $\C_Q$ as introduced in Definition \ref{deffree}.

\begin{proposition}[Sieving inequality] \label{prop:siev_ineq}
 Let $n,N$ be fixed  divisors of $Q$ and $r,R$ divisors of $Q/n,Q/N$, respectively.  
  Set
\[ N(r,R) :=\#\{ (x,y)\in\C_Q^2 \ : \ x\text{ is $(r,n)$-free and $y$ is $(R,N)$-free} \} . \]
For  choices  $p_1, \ldots, p_u$  of distinct prime divisors of $r$ and $l_1, \ldots l_v$ of distinct prime divisors of $R$,   write $r^*=k_rp_1\cdots p_u$ and  $R^*=k_Rl_1\cdots l_v$,
where $k_r$ and $k_R$ are also square-free.
 Then
 \begin{equation}\label{sieveeq}
N(r,R) \geq \sum_{i=1}^u N(k_rp_i,k_R)+\sum_{i=1}^v N(k_r,k_Rl_i)   - (u+v-1)N(k_r,k_R). 
\end{equation}
Further, set $\delta = 1-\sum_{i=1}^u 1/p_i -\sum_{i=1}^v1/l_i$.  Then $(\ref{sieveeq})$ can be expressed in the form
\begin{equation}\label{eq:sieveeq1}
\begin{split} 
N(r,R) \geq\delta N(k_r,k_R)+\sum_{i=1}^u\left(N(k_r p_i,k_R) -\left(1- \frac{1}{p_i}\right)N(k_r,k_R)\right) \\ +\sum_{i=1}^v\left(N(k_r,k_Rl_i) -\left(1- \frac{1}{l_i}\right)N(k_r,k_R)\right).
\end{split}
\end{equation}

\end{proposition}
\begin{proof} Each of the $N$-terms on the right side of (\ref{sieveeq}) can  only score (count one)  for pairs $(x,y) \in \C_Q^2$ for which $x$ is $(k_r,n)$-free and $y$ is $(k_R,N)$-free.    In addition,  to be scored, for instance, by $N(k_rp_i,k_R)$, $(x,y)$ has to be $(p_i,n)$-free.     We conclude that the aggregate score counted on the right side by members of the set in the definition of $N(r,R)$ is $u+v-(u+v-1)=1$.    On the other hand, if,  for instance $(x,y)$ is $(k_r,n)$-free  but not $(p_i,n)$-free and $y$ is $(k_R,N)$-free, it will not score in $N(k_rp_i)$ and so its aggregate score will be non-positive.
\end{proof}


We apply the sieving inequality with $\C_Q=\F_q^*$ to produce a sieving version of Theorem \ref{thm:main_cond}.  From Lemma \ref{lem:rn}  we could  assume $r,R$ are square-free (though some of their prime factors may also be divisors of $n, N$, respectively).
\begin{theorem} \label{thm:sieve_cond}
Assume the notation and conditions of Theorem $ \ref{thm:main_cond}$. Further, let $p_1, \ldots, p_u$  be distinct primes dividing $r$ and  $l_1, \ldots, l_v$ be distinct primes dividing
 $R$.  Write $r^*=k_rP_r$, where, for each $i=1, \ldots, u$, $p_i|P_r$ but $p_i\nmid k_r$ and similarly $R^*=k_RP_R$.  Set $ \delta =1 -\sum_{i=1}^u 1/p_i- \sum_{i=1}^v 1/l_i$ and suppose that $\delta >0$. 
Then
\begin{equation}\label{eq:Nbound}
N_{f, F}\geq \delta\cdot \frac{\phi(k_r) \phi(k_R)}{k_rnk_RN}\left (q- DnNW(k_r)W(k_R)\left(\frac{u+v-1}{\delta}+2\right)q^{1/2}\right).
\end{equation}
\end{theorem}

\begin{proof}
Assume that $r,R$ are square-free.   Given $n,F,n,N$, let $N(r,R)$ stand for  $N_{f,F}(r,n,R,N)$. Further, set $\theta= \frac{\phi(k_r)\phi(k_R)}{k_rnk_RN}$.
From Theorem \ref{thm:main_cond},
  \begin{equation} \label{eq:mainterm}
  N(k_r,k_R) \geq  \theta(q- DnNW(k_r)W(k_R)q^{1/2}). 
  \end{equation}
  
  Next, we bound the differences shown in (\ref{eq:sieveeq1}).  Towards this, for $1 \leq i \leq u$, set $\Delta_{p_i}= N (k_rp_i,k_R)-\left(1- \frac{1}{p_i}\right)N(k_r,k_R)$.  Then, with $G$ standing for $G_{f,F}$ in the proof of  Theorem \ref{thm:main_cond}, we have that
  
  \[\Delta_{p_i} =\theta\left(1-\frac{1}{p_i}\right) \left(\sum_{\substack{ t|k_rn\\ T|k_RN}}\frac{\mu((tp_i)_{(n)})\mu(T_{(N)})}{\phi((tp_i)_{(n)})\phi(T_{(N)})}\sum_{\substack{\ord(\eta)=tp_i\\ \ord(\chi)=T}}G(\eta,\chi)\right ) . \]
  Since each character $\eta$ is nontrivial, we have the bound
  \begin{eqnarray}\label{eq:diff1}
  |\Delta_{p_i}| & \leq \theta\left(1-\frac{1}{p_i}\right) DnN(W(k_rp_i)-W(k_r))W(k_R)q^{1/2}\nonumber\\
  & = \theta\left(1-\frac{1}{p_i}\right) D nNW(k_r)W(k_R)q^{1/2}.
\end{eqnarray}

Similarly, if  $\Delta_{l_i}= N (k_r,k_Rl_i)-\left(1- \frac{1}{l_i}\right)N(k_r,k_R)$  for each $1\leq i \leq v$,  we have
\begin{equation} \label{eq:diff2}
|\Delta_{l_i}| \leq \theta\left(1-\frac{1}{l_i}\right)  D nNW(k_r)W(k_R)q^{1/2}.
\end{equation}
Inserting the bounds (\ref{eq:mainterm}), (\ref{eq:diff1}) and (\ref{eq:diff2}) in  (\ref{eq:sieveeq1}), we obtain (\ref{eq:Nbound}).
\end{proof}

%
%
The next theorem is an immediate consequence of Theorem \ref{thm:sieve_cond}  when  $r=(q-1)/n$ and $R=(q-1)/N$.
\begin{theorem}\label{thm:main_siev}
Let $f,F,n,N$ be as in Theorem $\ref{thm:main_cond}$. Write $((q-1)/n)^* = k_n p_1 \cdots p_u$, where  $p_1,\ldots ,p_u$ are distinct primes and similarly $((q-1)/N)^*=k_N l_1 \cdots l_v$. Set $\delta = 1-\sum_{i=1}^u 1 /p_i- \sum_{i=1}^v 1/l_i$ and assume $\delta>0$.
Then, using the same notation as in Theorem~$\ref{thm:main_cond}$, there exists some $(x,X)\in\F_q^2$, such that $f(x)$ is $n$-primitive and $F(X)$ is $N$-primitive, provided that
\[ q^{1/2}  > DnNW(k_n)W(k_N)\left( \frac{u+v-1}{\delta} + 2 \right) . \]

\end{theorem}
We will refer to the primes $p_1,\ldots ,p_u, l_1, \ldots, l_v$ appearing above as the \emph{sieving primes}.

\section{Special points on elliptic curves}

In this section, we apply our methods to study special points on elliptic curves. More specifically, given an elliptic curve  $\mathcal C:y^2=f(x)$ defined over $\F_q$, with $f\in\F_q[x]$ being a square-free cubic, we study the existence of $\F_q$-primitive points on $\mathcal C$.

Equivalently, we request a primitive $x$, such that $f(x)$ is $2$-primitive, i.e., our goal is to prove that
\[ N_f := N_{x,f(x)}(q-1,1,(q-1)/2,2)  \] 
is positive.
Notice that $x, f(x)$ are square-free polynomials and the ratio $x/f(x)$ is not a constant. Thus Corollary~\ref{cor:main_cond} yields that a sufficient condition for $N_f>0$ is
\begin{equation} \label{eq:main_cond1}
q^{1/2} \geq 3\cdot 1\cdot 2\cdot W(q-1)W((q-1)/2) = 6 W(q-1)W\left( \frac{q-1}{2} \right) .
\end{equation}
%
%

Our next aim is to explore the numerical aspects of \eqref{eq:main_cond1}. Naturally, an estimate of the number $W(a)$ is necessary.
\begin{lemma} \label{lem:w}
Let $t$, $a$ be positive integers and let $p_1 , \ldots, p_j$ be the distinct prime divisors of
$t$ such that $p_i \leq 2^a$. Then $W (t) \leq c_{t,a} t^{1/a}$, where
\[  c_{t,a} = \frac{2^j}{( p_1 \cdots p_j )^{1/a}} . \]
In particular, for 
$d_t := c_{t,6}$ 
we have the bound 
$d_t<37.47$.
\end{lemma}
\begin{proof}
The statement is an immediate generalization of Lemma~3.3 of \cite{cohenhuczynska03} and can be proved
using multiplicativity. The bound for 
$d_t$ 
can be easily computed.
\end{proof}
Here  we have singled out the value $a=6$ for convenience in what follows. Now, we move on to numerical computations. We note that for this purpose we relied on the \textsc{SageMath} mathematics software system.

Notice if $(q-1)/2$ is even, then $W((q-1)/2) = W(q-1)$, whilst if $(q-1)/2$ is odd, then $W((q-1)/2) = W(q-1)/2$. It follows that a combination of \eqref{eq:main_cond1} and Lemma~\ref{lem:w} yields two conditions, depending on the parity of $(q-1)/2$. Namely, 
\[ q^{1/6} \geq \frac{6\cdot 37.47^2}{2^{1/3}} \quad \text{and} \quad q^{1/6} \geq 3\cdot 37.47^2  , \]
if $q\equiv 1\pmod{4}$ and $q\equiv 3\pmod{4}$ respectively. We check the stronger of the two conditions, i.e., the former, and verify that it is satisfied for $q\geq 8.94\cdot 10^{22}$.

Then, we observe that if $q-1$ is divided by $18$ or more prime numbers, then $q> 8.94\cdot 10^{22}$, which implies that the case $W(q-1)\geq 2^{18}$ is settled.

Our next step will settle the cases $2^{13}\leq W(q-1)\leq 2^{17}$. Let $t_{\min}\leq t_{\max}$ be two positive integers. Further, let $p_i$ stand for the $i$-th prime number, that is, for example, $p_1=2$ and $p_3=5$. Now, assume that for some odd prime power $q$, we have that $2^{t_{\min}}\leq W(q-1)\leq 2^{t_{\max}}$. It follows that $q-1$ is divided by at least $t_{\min}$ and by at most $t_{\max}$ prime numbers, which implies that $q>p_1\cdots p_{t_{\min}}$. Moreover, we may choose as sieving primes the largest $n< t_{\min}$ prime divisors of $q-1$ and, more precisely, take each of them twice, in such a way that $1-\sum_{i=1}^n \frac{2}{p_i} > 0$. This way we ensure that the number $\delta$ (see Theorem~\ref{thm:sieve_cond}) satisfies $\delta > 1-\sum_{i=1}^n \frac{2}{p_i} > 0$. It then follows from Theorem~\ref{thm:sieve_cond} that $N_f>0$ if
\[ \sqrt{p_1\cdots p_{t_{\min}}} > 6\cdot 4^{t_{\max} - n} \left( \frac{2n-1}{1-\sum_{i=1}^n \frac{2}{p_i}} + 2 \right) . \]
We computationally verify that the above holds for the pairs $(t_{\max},t_{\min}) = (17,15)$ and $(14,13)$, that is, the case $W(q-1)\geq 2^{13}$ is settled, thus, the case $q>6\cdot 4^{12} = 100663296$ is settled.

We proceed to reducing the number of possible exceptions as much as possible. First, we try the condition of Corollary~\ref{cor:main_cond} within the range $3<q\leq 100663296$ with each quantity explicitly computed, instead of using generic estimates. It turns out that within this range there are 5798811 odd prime powers, with 797566 of them failing to satisfy this condition and $q=100663291$ being the largest among them.

Finally, we attempt to use the prime sieve, see Theorem~\ref{thm:sieve_cond}, on these 
 persistent prime powers, again with all the quantities explicitly computed. Our computations reveal that there exists a suitable set of sieving primes for almost 97\% of these numbers. More precisely, this method was unsuccessful for 24826 out of the 797566 prime numbers checked, with $q=82192111$ being the largest among them. To sum up, we have proved the following.
\begin{theorem} \label{thm:elliptic1}
  Let $q>82192111$ be an odd prime power. Further, let $f(x)\in\F_q[x]$ be a square-free polynomial of degree $3$, then the elliptic curve $\mathcal C : y^2=f(x)$ contains $\F_q$-primitive points.
\end{theorem}
We believe that identifying the genuine exceptions to the above theorem is an interesting research question.
\begin{problem}
Identify all the odd prime powers $3\leq q\leq 82192111$ and the corresponding square-free cubic polynomials $f\in\F_q[x]$ such that the elliptic curve $\mathcal C : y^2=f(x)$ does not contain $\F_q$-primitive points.
\end{problem}

\subsection{The elliptic curve $\mathcal C : y^2=x^3-ax$}

Finally, we study the special case of the elliptic curve $\mathcal C : y^2 = f_a(x)$, where $f_a(x)= x^3-ax$, $a\in\F_q^*$. Notice that since $f_a(0)=0$, the polynomial $x\cdot f_a(x)$ has 3 distinct roots, so, in this special case, the condition of \eqref{eq:main_cond1} may be replaced by the significantly weaker condition
\[q^{1/2} \geq 2\cdot 1\cdot 2\cdot W(q-1)W((q-1)/2) = 4 W(q-1)W\left( \frac{q-1}{2} \right) . \]

We repeat the same steps that lead us to Theorem~\ref{thm:elliptic1}. Having a weaker condition, we obtain that if $q>16763671$, then the elliptic curve $\mathcal C:y^2=f_a(x)$ over $\F_q$, always has some $\F_q$-primitive point, while in the range $3\leq q\leq 16763671$ there are exactly 11041 odd prime powers that may not possess this property. Additionally, observe that, so far, the same can be said for any square-free cubic $g\in\F_q[x]$ with $g(0)=0$.

We return to our special case and attempt an exhaustive search. In particular, we attempt to identify explicitly a point on the curve $\mathcal C:y^2=f_a(x)$, for all the 11041 persistent prime powers $q$ and for all $a\in\F_q^*$. Due to hardware restrictions and the vast number of elliptic curves that have to be checked, this computation was not completed. In particular, we checked the first 4624 prime powers, which leaves us with 6417 prime powers, all within the range $141121\leq q\leq 16763671$ unchecked. These partial results suggest that all of these elliptic curves have some $\F_q$-primitive point, with the only exceptions being $q=3$, $5$, $7$, $9$, $13$, $17$, $25$, $29$, $31$, $41$, $49$, $61$, $73$, $81$, $121$ and $337$, while the number of curves over the corresponding finite fields with no such points is given in Table~\ref{tab:1}.
\begin{table}[h]
\begin{center}
    \begin{tabular}{c|cccccccccccccccc}
      $q$ & $3$& $5$& $7$& $9$& $13$& $17$& $25$& $29$& $31$& $41$& $49$& $61$& $73$& $81$& $121$& $337$ \\ \hline
      \# curves & 1 & 2 & 3 & 5 & 5 & 6 & 12 & 1 & 1 & 8 & 8 & 10 & 12 & 10 & 16 & 2
    \end{tabular}
\end{center}
    \caption{Number of curves $\mathcal C:y^2=x^3-ax$, $a\in\F_q^*$, over $\F_q$, without $\F_q$-primitive points, when $q\not\in [141121, 16763671]$.\label{tab:1}}
\end{table}
These findings enable us to state Conjecture~\ref{conj:ell2}.
We believe that attacking Conjecture~\ref{conj:ell2} in its full generality would be nontrivial but not hopeless: it would require a combination of advanced theoretical techniques with exhaustive computational methods.

Finally, we repeat the same procedure for two special curves within the aforementioned family of curves, namely the curves $\mathcal C:y^2=x^3-x$ and $\mathcal C:y^2=x^3+x$. In particular, after spending just a few seconds of computer time, we explicitly check all the possibly exceptional curves and, as a result, we obtain Theorem~\ref{thm:ell1}.
%
%
%

\end{document}